\documentclass[11pt]{article}
\usepackage[T1]{fontenc}
\usepackage{amsfonts,amsmath,tikz}
\usepackage{epsfig}
\usepackage{latexsym}
\usepackage{color}
\usepackage{soul}  
\usepackage{epsf,amssymb}

\usepackage{amsthm}
\usepackage{yfonts}
\usepackage{hyperref}
\usepackage{ifthen}
\usepackage{enumitem}
\usepackage{subcaption}

\usepackage{float} 
\floatstyle{plaintop}
\restylefloat{table} 

\newtheorem{theorem}{Theorem}
\newtheorem{lemma}[theorem]{Lemma}

\newtheorem{conjecture}[theorem]{Conjecture}

\newtheorem{corollary}[theorem]{Corollary}
\newtheorem{proposition}[theorem]{Proposition}

\def\vertex(#1){\put(#1){\circle*{2}}}
\def\vertexo(#1){\put(#1){\circle{2}}}
\def\vert(#1){\put(#1){\circle*{1.5}}}
\def\verto(#1){\put(#1){\circle{1.5}}}
\def\lab(#1)#2{\put(#1){\makebox(0,0)[c]{#2}}}
\definecolor{darkgreen}{RGB}{0,100,0}

\newcommand{\gammakL}{\gamma_k^{\mathrm L}}

\tikzset{My Style/.style={draw, circle, fill=black,scale=0.5}} 

\tikzset{My Style2/.style={draw, rectangle, fill=white, minimum size=1.5mm}} 

\title{On $k$-limited domination: complexity and Sierpi\'nski graphs}

\author{
Dragana Bo\v zovi\'{c} \\
\small \it University of Maribor, FEECS, Koroška cesta 46, 2000 Maribor, Slovenia \\
\small \it University of Maribor, FNM, Koroška cesta 160, 2000 Maribor \\
\small \tt dragana.bozovic@um.si
}
\date{}

\begin{document}
\maketitle

\begin{abstract}
A dominating set $D$ of a graph $G$ is called $k$-limited if every vertex of $D$ has at most $k$ neighbors outside $D$. The minimum cardinality among all $k$-limited dominating sets of $G$ is the $k$-limited domination number, denoted by $\gamma_k^L(G)$.

In this paper, we prove that the \textsc{$1$-Limited Dominating Set} problem is $\mathsf{NP}$-complete, answering an open question posed in the literature. We further show that, for every positive integer $k$, the \textsc{$k$-Limited Dominating Set} problem is $\mathsf{NP}$-complete even when restricted to planar graphs. In addition, we study the $k$-limited domination number of Sierpi\'nski graphs. We determine the exact value of $\gamma_1^L(S(n,m))$ for all integers $n\ge 1$ and $m\ge 2$, and obtain the $k$-limited domination number for the planar family $S(n,3)$.

\end{abstract}



\section{Introduction}

Domination is one of the central topics in graph theory and has been extensively studied over the past decades. Besides the classical domination number, numerous variants of domination have been introduced and investigated, including total domination, efficient domination, Roman domination, Italian domination, and many others.
For a comprehensive overview of domination in graphs, we refer the reader to the monographs \cite{HHH1, HHH2, HHS1}.


The concept of $k$-limited domination was introduced in \cite{skupni}. In that paper, the authors established sharp general bounds for $\gamma_k^L(G)$ and discussed their tightness. They also obtained exact values for several basic graph families, including paths, stars, complete graphs, and complete bipartite graphs. Furthermore, they related $k$-limited domination to $(1,t)$-domination and showed that, for every $d$-regular graph $G$ and every $k<d$, the equality $\gamma_k^L(G)=\gamma_{1,d-k}(G)$ holds. 
For $k=1$, they connected $1$-limited domination with the packing number and proved the Gallai-type inequality $\gamma_1^L(G)+\rho(G)\le |V(G)|$.

In \cite{sandra}, the author proved that the \textsc{$k$-Limited Dominating Set} problem is $\mathsf{NP}$-complete for every fixed integer $k\ge 2$. The paper also considered the $k$-limited domination number of Cartesian products of graphs, establishing general lower and upper bounds for $\gamma_k^L(G\square H)$ and deriving exact values for several families of graph products, including rook graphs, Cartesian products of $k$-coronas, grid graphs, and several families involving prisms and hypercubes. In addition, the author posed the question of whether the \textsc{$k$-Limited Dominating Set} problem is $\mathsf{NP}$-complete for $k=1$. In this paper, we provide an answer to this question.

Besides complexity questions, an important aspect of the study of $k$-limited domination is the determination of exact values of the parameter for specific graph families. In this context, we focus on Sierpi\'nski graphs, one of the most extensively studied recursively defined graph families. Sierpi\'nski graphs were introduced by Klavžar and Milutinović \cite{klavzarmilutinovic} in connection with a variant of the Tower of Hanoi problem. Several domination-related parameters have already been investigated on these graphs, including the domination number and efficient domination \cite{dominacijasierpinski}, Grundy domination \cite{grundysierpinski}, (double) Roman domination \cite{romansierpinski}, and (perfect) Italian domination \cite{italiansierpinski}. We study the $k$-limited domination number of Sierpi\'nski graphs.

The paper is organized as follows. Section~2 introduces the notation and terminology used throughout the paper. In Section~3, we prove that the $1$-limited dominating set problem is $\mathsf{NP}$-complete, while in Section~4 we show that the problem remains $\mathsf{NP}$-complete even when restricted to planar graphs. Section~5 is devoted to Sierpi\'nski graphs, where we determine the $1$-limited domination number of $S(n,m)$ and the $k$-limited domination number of the planar family $S(n,3)$. Finally, in Section~6, we present some concluding remarks and directions for future research.

\section{Preliminaries}

Throughout the paper, all graphs are finite, simple, and undirected. Let $G=(V(G),E(G))$ be a graph. For a vertex $v\in V(G)$, we denote by $N(v)$ its open neighborhood and by $N[v]=N(v)\cup\{v\}$ its closed neighborhood. For a positive integer $n$, we write $[n]=\{1,\ldots,n\}$.

A set $D\subseteq V(G)$ is a \emph{dominating set} of $G$ if every vertex of $V(G)-D$ has a neighbor in $D$. The minimum cardinality 
among all dominating sets of $G$ is the \emph{domination number} of $G$, denoted by $\gamma(G)$.

Let $k$ be a positive integer. A dominating set $D\subseteq V(G)$ is called a \emph{$k$-limited dominating set} if every vertex $u\in D$ satisfies $|N(u)-D|\le k.$
The minimum cardinality among all $k$-limited dominating sets of $G$
is the \emph{$k$-limited domination number} of $G$, denoted by $\gamma_k^L(G)$. A $k$-limited dominating set of cardinality $\gamma_k^L(G)$ is called a \emph{$\gamma_k^L(G)$-set}. For convenience, we refer to the condition $|N(u)-D|\le k$ as the \emph{$k$-limited constraint}.

We also recall the definition of Sierpi\'nski graphs. For integers $n\ge 1$ and $m\ge 2$, the vertex set of the Sierpi\'nski graph $S(n,m)$ is $V(S(n,m))=\{0,1,\ldots,m-1\}^n$.
Two distinct vertices $u=u_1u_2\dots u_n$ and $v=v_1v_2\dots v_n$
are adjacent if and only if there exists an index $h\in \{1,\ldots,n \}$ such that
\begin{enumerate}
\item $u_t=v_t$ for every $t<h$,
\item $u_h\neq v_h$,
\item $u_t=v_h$ and $v_t=u_h$ for every $t>h$.
\end{enumerate}
It is well known that $S(1,m)\cong K_m$ and that $S(n,m)$ consists of $m$ copies of $S(n-1,m)$, where every pair of copies is connected by exactly one \textit{linking edge}. We refer to each copy of $S(1,m)$ in $S(n,m)$ as a \emph{base clique}.

\section{Computational complexity of the $1$-limited dominating set problem}

In \cite{sandra}, the author posed the question whether the
\textsc{$1$-Limited Dominating Set} problem is $\mathsf{NP}$-complete. In this section, we show that this is indeed the case.

We prove that the \textsc{1-Limited Dominating Set} problem is $\mathsf{NP}$-complete by a reduction from the \textsc{Vertex Cover} problem, which is known to be $\mathsf{NP}$-complete \cite{karp}. 

\begin{center}
\fbox{\parbox{0.8\linewidth}{\noindent
\textbf{\textsc{$1$-Limited Dominating Set problem}}\\[.8ex]
\begin{tabular*}{.93\textwidth}{rl}
{\em Instance:} & A graph $G$ and a positive integer $\ell$.\\
{\em Question:} & Does $G$ admit a 1-limited dominating set of size at most $\ell$?
\end{tabular*}
}}
\end{center}

\begin{center}
\fbox{\parbox{0.8\linewidth}{\noindent
\textbf{\textsc{Vertex Cover problem}}\\[.8ex]
\begin{tabular*}{.93\textwidth}{rl}
{\em Instance:} & A graph $G$ and a positive integer $\ell$.\\
{\em Question:} & Does $G$ admit a vertex cover of size at most $\ell$?
\end{tabular*}
}}
\end{center}

The reduction is constructed as follows.
Let $G$ be a graph with vertex set $V(G)=\{ v_1, \dots, v_n\}$. The graph $G'$ is obtained from $G$ in the following way.
For every $v_iv_j\in E(G)$, replace the edge $v_iv_j$ by a copy of the graph $G^{ij}$, called an \textit{edge gadget}, with vertex set
$$V(G^{ij})=\{x_{ij}, y_{ij}, z_{ij}, a_{ij}, b_{ij}, c_{ij}\}$$
and edge set
$$E(G^{ij})=\{a_{ij}b_{ij}, a_{ij}c_{ij}, a_{ij}x_{ij}, a_{ij}y_{ij}, b_{ij}c_{ij}, b_{ij}y_{ij}, c_{ij}y_{ij}, c_{ij}z_{ij}, x_{ij}y_{ij}, y_{ij}z_{ij}\}.$$

Finally, for every $v_iv_j\in E(G)$, add edges $v_ia_{ij}$ and $c_{ij}v_j$, see Figure~\ref{edge_gadget}.

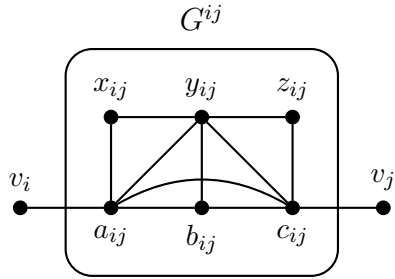
\begin{figure}[ht!]
\begin{center}
\begin{tikzpicture}[scale=0.6]

\node [My Style, name=vi, label=above:$v_i$ ]   at (-2,0) {};		
\node [My Style, name=x1, label=below:$a_{ij}$ ]   at (0,0) {};		
\node [My Style, name=x2, label=below:$b_{ij}$]   at (2,0) {};	
\node [My Style, name=x3, label=below:$c_{ij}$]    at (4,0) {};	
\node [My Style, name=y1, label=above:$x_{ij}$]   at (0,2) {};	
\node [My Style, name=y2, label=above:$y_{ij}$ ]   at (2,2) {};	
\node [My Style, name=y3, label=above:$z_{ij}$]   at (4,2) {};
\node [My Style, name=vj, label=above:$v_j$ ]   at (6,0) {};		
		
\draw[thick] (vi)--(x1)--(x2)--(x3)--(vj);
\draw[thick] (y1)--(y2)--(y3);
\draw[thick] (x1)--(y1);
\draw[thick] (x2)--(y2);
\draw[thick] (x3)--(y3);
\draw[thick] (x1)--(y2);
\draw[thick] (x3)--(y2);

\draw[thick, bend left] (x1) to (x3);

\draw[rounded corners=10pt, thick] (-1,-1.5) rectangle (5,3.5);
\node[font=\large\bfseries] at (2,4.2){$G^{ij}$};


\end{tikzpicture}
\end{center}
\caption{The edge gadget $G^{ij}$.}
\label{edge_gadget}
\end{figure}

The resulting graph $G'$ has order $n+6m$ and can clearly be constructed in polynomial time.\\

We begin with a lemma that will later be used in the proof.

\begin{lemma}\label{5ali6}
Let $D\subseteq V(G')$ be a $1$-limited dominating set of $G'$, $v_iv_j\in E(G)$, and let $D_{ij}=D\cap V(G^{ij})$. Then $|D_{ij}|\ge 5$. Furthermore, if $v_i,v_j\notin D$, then $|D_{ij}|\ge 6$.
\end{lemma}

\begin{proof}
Let $v_iv_j\in E(G)$ be fixed and let $D_{ij}=D\cap V(G^{ij})$.
We first show that $|D_{ij}|\ge 5$. Suppose, to the contrary, that $|D_{ij}|\le 4$. Then at least two vertices of $G^{ij}$ lie outside $D_{ij}$. Let $R=V(G^{ij})-D_{ij}$. Thus $|R|\ge 2$.

If $y_{ij}\notin R$, then $y_{ij}\in D_{ij}$, and since $y_{ij}$ is adjacent to every other vertex of $G^{ij}$, it has at least two neighbors in $R$. Hence, $y_{ij}$ violates the $1$-limited constraint. Therefore, $y_{ij}\in R$. 
Since $y_{ij}\in R$, we claim that $a_{ij}\in D_{ij}$. Suppose, to the contrary, that $a_{ij}\in R$. If $x_{ij}\in D_{ij}$, then $x_{ij}$ has two neighbors outside $D$, namely $a_{ij}$ and $y_{ij}$, contradicting the $1$-limited constraint. If $x_{ij}\notin D_{ij}$, then $x_{ij}$ is not dominated, since $N(x_{ij})=\{a_{ij},y_{ij}\}$. Thus $a_{ij}\in D_{ij}$.
Since $a_{ij}\in D_{ij}$ and $y_{ij}\in R$, all other neighbors of $a_{ij}$ in $G^{ij}$, namely $b_{ij}, c_{ij}, x_{ij}$, must belong to $D_{ij}$, otherwise $a_{ij}$ would have at least two neighbors outside $D$, again contradicting the $1$-limited constraint. 
Now $c_{ij}\in D_{ij}$ and $y_{ij}\in R$, meaning that $z_{ij}\in D_{ij}$, otherwise $c_{ij}$ would have two neighbors outside $D$. 
Consequently, $\{a_{ij},b_{ij},c_{ij},x_{ij},z_{ij}\}\subseteq D_{ij}$, and therefore $|D_{ij}|\ge 5$.

Suppose now that $v_i,v_j\notin D$. We show that $|D_{ij}|\ge 6$.
By the first part, $|D_{ij}|\ge 5$. Suppose, to the contrary, that
$|D_{ij}|=5$, and let $w$ be the unique vertex of $V(G^{ij})-D_{ij}$. If $w\in\{b_{ij},x_{ij},y_{ij},c_{ij}\}$, then necessarily $a_{ij}\in D_{ij}$, and $a_{ij}$ has two neighbors outside $D$, namely $v_i$ and $w$, contradicting the $1$-limited constraint. Similarly, if $w\in\{a_{ij},b_{ij},y_{ij},z_{ij}\}$, then necessarily $c_{ij}\in D_{ij}$, 
and $c_{ij}$ has two neighbors outside $D$, namely $v_j$ and $w$, a contradiction again. Therefore, the case $|D_{ij}|=5$ is impossible. Since $|D_{ij}|\ge 5$, it follows that
$|D_{ij}|\ge 6$, and the proof is complete.
\end{proof}

We are now ready to prove the main proposition of the reduction.
\begin{proposition}
Let $G$ be a graph of order $n$, and let $G'$. Then $G$ has a vertex cover of size at most $\ell$ if and only if $G'$ has a $1$-limited dominating set of size at most $\ell +5m$.
\end{proposition}

\begin{proof}
Let $C\subseteq V(G)$ be a vertex cover of $G$ with $|C|\le \ell$. We construct a set $D\subseteq V(G')$ as follows. First, put all vertices of $C$ into $D$. For every edge $v_iv_j\in E(G)$, add the vertices $a_{ij}, b_{ij}, c_{ij}, y_{ij}$ to $D$. Since $C$ is a vertex cover of $G$, at least one of $v_i$ and $v_j$ belong to $C$. 
If $v_i\in C$ and $v_j\notin C$, add $z_{ij}$ to $D$. If
$v_i\notin C$ and $v_j\in C$, add $x_{ij}$ to $D$.
If both $v_i$ and $v_j$ belong to $C$, add exactly one of $x_{ij}$ and $z_{ij}$ to $D$.
Thus, exactly five vertices are chosen from each edge gadget, and hence $|D|=|C|+5m\le \ell +5m$. One easily verifies that $D$ is a dominating set of $G'$ and that every vertex of $D$ has at most one neighbor outside $D$.
Therefore, $D$ is a $1$-limited dominating set of $G'$ of size at most $\ell+5m$.

Conversely, let $D$ be a $1$-limited dominating set of $G'$ with $|D|\le \ell+5m$. For every edge $v_iv_j\in E(G)$, let $D_{ij}=D\cap V(G^{ij})$. By Lemma~\ref{5ali6}, it holds that $|D_{ij}|\ge 5$. Let $C'=\{v_i \; ; \; v_iv_j\in E(G) \text{ and } |D_{ij}|\ge 6\}$ and define $C:=(D\cap V(G))\cup C'$. That is, for every edge gadget containing at least six vertices of $D$, we
add one endpoint of the corresponding edge to $C$. We claim that $C$ is a vertex cover of $G$. Let $v_iv_j\in E(G)$. If $v_i$ or $v_j$ belongs to $D\cap V(G)$, then $v_iv_j$ is covered by $C$. Otherwise, $v_i,v_j \notin D\cap V(G)$, and by Lemma~\ref{5ali6}, we have $|D_{ij}|\ge 6$. Hence, $v_i\in C'\subseteq C$, and so $v_iv_j$ is again covered by $C$. Therefore, $C$ is a vertex cover of $G$. Finally, $$|C|\le |D\cap V(G)|+\sum_{v_iv_j\in E(G)}(|D_{ij}|-5)=|D|-5m\le \ell.$$
Thus $G$ has a vertex cover of size at most $\ell$, completing the proof.
\end{proof}

The \textsc{1-Limited Dominating Set} problem belongs to $\mathsf{NP}$, since for a given set $D$ we can verify in polynomial time that $D$ is a dominating set and that it satisfies the $1$-limited constraint.
Since the \textsc{Vertex Cover} problem is $\mathsf{NP}$-complete, this yields the following theorem.

\begin{theorem} 
\label{complex}
The \textsc{$1$-Limited Dominating Set} problem is $\mathsf{NP}$-complete.
\end{theorem}

Furthermore, the reduction preserves planarity. If $G$ is planar, then so is $G'$, since every edge of $G$ is replaced by the gadget shown in Figure~\ref{edge_gadget}, which admits a planar embedding. Since the \textsc{Vertex Cover} problem remains $\mathsf{NP}$-complete even when restricted to planar graphs~\cite{gareyjohnson}, we obtain the following corollary.

\begin{corollary}\label{cor:NP-planar-k1}
The \textsc{$1$-Limited Dominating Set} problem is $\mathsf{NP}$-complete even when restricted to planar graphs.
\end{corollary}

\section{Computational complexity of the $k$-limited dominating set problem on planar graphs}

In the previous section we proved that the \textsc{$1$-Limited Dominating Set} problem is $\mathsf{NP}$-complete even when restricted to planar graphs. We now consider the case $k\ge 2$.

It is known that the \textsc{$k$-Limited Dominating Set} problem is $\mathsf{NP}$-complete for every fixed integer $k\ge2$~\cite{sandra}. Moreover, it is straightforward from the construction in \cite{sandra} to verify that, for $2\le k\le 4$, the reduction also preserves planarity. Together with Corollary~\ref{cor:NP-planar-k1}, this yields the following corollary.

\begin{corollary}\label{cor:NP-planar-k234}
For every $1\le k \le 4$, the \textsc{$k$-Limited Dominating Set} problem is $\mathsf{NP}$-complete even when restricted to planar graphs.
\end{corollary}

It therefore remains to establish the hardness result for $k\ge 5$.
To this end, we reduce from the \textsc{Dominating Set} problem restricted to planar graphs of maximum degree at most $3$, which is known to be $\mathsf{NP}$-complete \cite{gareyjohnson-knjiga}. Let $G$ be such a graph with vertex set $V(G)=\{ v_1, \dots, v_n\}$ and let $\ell$ be a positive integer.

\begin{center}
\fbox{\parbox{0.8\linewidth}{\noindent
\textbf{\textsc{$k$-Limited Dominating Set problem}}\\[.8ex]
\begin{tabular*}{.93\textwidth}{rl}
{\em Instance:} & A graph $G$ and a positive integer $\ell$.\\
{\em Question:} & Does $G$ admit a $k$-limited dominating set of size at most $\ell$?
\end{tabular*}
}}
\end{center}

\begin{center}
\fbox{\parbox{0.8\linewidth}{\noindent
\textbf{\textsc{Dominating Set problem}}\\[.8ex]
\begin{tabular*}{.93\textwidth}{rl}
{\em Instance:} & A planar graph $G$ of maximum degree at most $3$ and a \\& positive integer $\ell$.\\
{\em Question:} & Does $G$ admit a dominating set of size at most $\ell$?
\end{tabular*}
}}
\end{center}

We construct a new graph $G'$ as follows.
For every $v\in V(G)$ and every $r\in[k-3]$, take a copy of the graph $H_v^r$, with vertex set
$$V(H_v^r)=\{p_{v,r},q_{v,r},s_{v,r}^1, \dots, s_{v,r}^{k-1} \}$$
and edge set
$$E(H_v^r)=\{p_{v,r}q_{v,r} \} \cup \{q_{v,r}s_{v,r}^i \, ; \; i\in [k-1] \}.$$
Finally, for every $v\in V(G)$ add edges $vp_{v,r}$ for $r\in[k-3]$, see Figure~\ref{vertex_gadget}.

\begin{figure}[ht!]
\begin{center}
\begin{tikzpicture}[scale=0.6]

\node [My Style, name=v, label=below:$v$ ]   at (0,-0.5) {};		
\node [My Style, name=p1, label=left:$p_{v,1}$]   at (-1.5,1) {};
\node [My Style, name=pz, label=right:$p_{v,{k-3}}$]    at (1.5,1) {};	

\node [My Style, name=q1, label=left:$q_{v,1}$]   at (-3,2.5) {};	
\node at (0,2.5) {$\dots$};
\node [My Style, name=qz, label=right:$q_{v,{k-3}}$]   at (3,2.5) {};

\draw[thick] (v)--(p1);
\draw[thick] (v)--(pz);
\draw[thick] (q1)--(p1);
\draw[thick] (qz)--(pz);

\node [My Style, name=s11, label=above:$s_{v,1}^1$]   at (-4,4) {};
\node at (-3,4) {$\dots$};
\node [My Style, name=s1z, label=above:$s_{v,1}^{k-1}$]   at (-2,4) {};
\draw[thick] (s11)--(q1);
\draw[thick] (s1z)--(q1);

\node [My Style, name=sz1, label=above:$s_{v,{k-3}}^1$]   at (2,4) {};
\node at (3,4) {$\dots$};
\node [My Style, name=szz, label=above:$s_{v,{k-3}}^{k-1}$]   at (4,4) {};
\draw[thick] (sz1)--(qz);
\draw[thick] (szz)--(qz);



\end{tikzpicture}
\end{center}
\caption{Vertex $v\in V(G)$ with $k-3$ pairwise disjoint copies of the graph $H_v^r$.}
\label{vertex_gadget}
\end{figure}
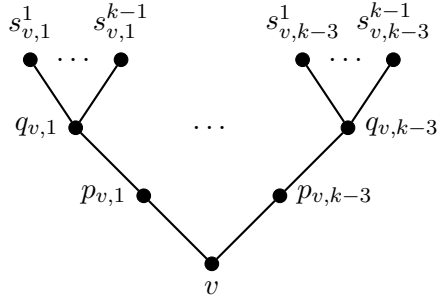

The resulting graph $G'$ has order $n+(k+1)(k-3)n$ and can clearly be constructed in polynomial time. Since each graph $H_v^r$ is attached locally to a single vertex of $G$, and since all these graphs are trees, the graph $G'$ is planar whenever $G$ is planar.

In the following lemma, we first establish some basic properties of $k$-limited dominating sets in the graphs $H_v^r$. 

\begin{lemma}\label{1alivec}
Let $k\ge 5$, let $H_v^r$ be one of the graphs attached to a vertex $v\in V(G)$, and let $D$ be a $k$-limited dominating set of $G'$. Then $|D\cap V(H_v^r)|\ge 1$. Moreover, if $|D\cap V(H_v^r)|=1$, then $D\cap V(H_v^r)=\{ q_{v,r}\}$. Furthermore, if $p_{v,r}\in D$, then $|D\cap V(H_v^r)|\ge 2$.
\end{lemma}

\begin{proof}
Let $H_v^r$ be fixed. First suppose that $D\cap V(H_v^r)=\emptyset$. Then none of the vertices $s_{v,r}^i$ belong to $D$. Since each $s_{v,r}^i$ is adjacent only to $q_{v,r}\notin D$, the vertices $s_{v,r}^i$ are not dominated, a contradiction. Hence $|D\cap V(H_v^r)|\ge 1$. 

Now suppose that $|D\cap V(H_v^r)|=1$. If the unique vertex in $D\cap V(H_v^r)$ was $p_{v,r}$, then none of the vertices $s_{v,r}^i$ would be dominated. If the unique vertex in $D\cap V(H_v^r)$ was any of the vertices $s_{v,r}^i$, say $s_{v,r}^j$,  then all the other $s_{v,r}^i$, where $i\neq j$, would not be dominated, since they are adjacent only to $q_{v,r}\notin D$. Therefore, $D\cap V(H_v^r)=\{ q_{v,r}\}$. 

Finally, suppose that $p_{v,r}\in D$. If $|D\cap V(H_v^r)|=1$, then the unique vertex in $D\cap V(H_v^r)$ is $p_{v,r}$, a contradiction. Hence, $|D\cap V(H_v^r)|\ge 2$.
\end{proof}

We are now ready to prove the main proposition of the reduction.

\begin{proposition}\label{NP-ravninski}
Let $k\ge 5$, let $G$ be a planar graph of order $n$ with maximum degree at most $3$. Then $G$ has a dominating set of size at most $\ell$ if and only if $G'$ has a $k$-limited dominating set of size at most $\ell +(k-3)n$.
\end{proposition}

\begin{proof}
Suppose first that $G$ has a dominating set $D$ with $|D|\le \ell$. We construct a set $D'\subseteq V(G')$ as 
$$D':=D\cup \{q_{v,r} \; ; \; v\in V(G), \ r\in[k-3] \}.$$ 
Then $|D'|=|D|+(k-3)n\le \ell+(k-3)n$. We show that $D'$ is a $k$-limited dominating set of $G'$. 
Since $D$ dominates $G$, all original vertices are dominated by $D'$. Moreover, for every graph $H_v^r$, the vertex $q_{v,r}$ dominates $p_{v,r}$ and all vertices $s_{v,r}^i$. Hence, $D'$ is a dominating set of $G'$. It remains to verify the $k$-limited constraint. 
Let $u\in D'$. If $u\in D\subseteq V(G)$, then $u$ has at most 3 neighbors in $G$, since $\Delta(G)\le 3$, and exactly $k-3$ neighbors $p_{u,r}$. Therefore $|N_{G'}(u)-D'|\le 3+(k-3)=k$. If $u=q_{v,r}$ for some $v\in V(G)$ and $r\in [k-3]$, then $N_{G'}(u)-D'=\{p_{v,r} \}\cup \{s_{v,r}^i \, ; \; i\in [k-1] \}$, so $|N_{G'}(u)-D'|=k$.
Thus $D'$ is a $k$-limited dominating set of $G'$.

Conversely, suppose that $G'$ has a $k$-limited dominating set $D'$ with $|D'|\le \ell+(k-3)n$. Let $A=\{v\in V(G) \; ; \; v\notin N_G[D'\cap V(G)] \}$. That is, $A$ consists of all the original vertices that are not dominated by $D'\cap V(G)$ in $G$. We claim that $D:=(D'\cap V(G))\cup A$ is a dominating set of $G$ of size at most $\ell$. First, $D$ is clearly a dominating set of $G$. Indeed, if $v\in A$, then $v\in D$. Otherwise, $v\notin A$, and by the definition of $A$, the vertex $v$ belongs to $N_G[D'\cap V(G)]$, and hence is dominated by $D$. It remains to estimate the size of $D$. By Lemma~\ref{1alivec}, every graph $H_v^r$ contains at least one vertex from $D'$. Since there are exactly $k-3$ such graphs for every $v\in V(G)$, these vertices contribute at least $(k-3)n$ vertices to $D'$. Now let $v\in A$. Since $v$ is not dominated by $D'\cap V(G)$ in $G$, neither $v$ nor any neighbor of $v$ in $G$ belongs to $D'\cap V(G)$. Since $D'$ dominates $G'$, vertex $v$ must be dominated by one of the vertices $p_{v,r}$ for some $r\in[k-3]$. Hence, for every $v\in A$ there exists $r\in [k-3]$ such that $p_{v,r}\in D'$.
By Lemma~\ref{1alivec}, each such graph $H_v^r$ contains at least $2$ vertices of $D'$, whereas every graph $H_v^r$ contributes at least one vertex to $D'$. Hence, every vertex of $A$ contributes one additional vertex to $D'$ beyond the mandatory contribution.
Consequently, $|A|\le |D'|-|D'\cap V(G)|-(k-3)n$, and therefore
\begin{align*}
|D|
&= |(D'\cap V(G))\cup A| \le |D'\cap V(G)|+|A| \\
&\le |D'\cap V(G)|+|D'|-|D'\cap V(G)|-(k-3)n \\
&= |D'|-(k-3)n \le \ell.
\end{align*}
Thus, $G$ has a dominating set of size at most $\ell$, completing the proof.
\end{proof}

The \textsc{$k$-Limited Dominating Set} problem belongs to $\mathsf{NP}$, since for a given set $D$ we can verify in polynomial time that $D$ is a dominating set and that it satisfies the $k$-limited constraint.
Since the \textsc{Dominating Set} problem is $\mathsf{NP}$-complete for planar graphs of maximum degree at most $3$, Proposition~\ref{NP-ravninski} yields the following theorem.

\begin{theorem} \label{NP-ravninski-k5}
For every $k\ge 5$, the \textsc{$k$-Limited Dominating Set} problem is $\mathsf{NP}$-complete even when restricted to planar graphs.
\end{theorem}

Combining Corollaries~\ref{cor:NP-planar-k1} and~\ref{cor:NP-planar-k234} with Theorem~\ref{NP-ravninski-k5}, we obtain the following result.

\begin{corollary}
For every integer $k\ge 1$, the \textsc{$k$-Limited Dominating Set} problem is $\mathsf{NP}$-complete even when restricted to planar graphs.
\end{corollary}

\section{The $k$-limited domination number of Sierpi\'nski graphs}

In the previous section, we proved that the \textsc{$k$-Limited Dominating Set} problem remains $\mathsf{NP}$-complete even when restricted to planar graphs. It is therefore natural to investigate the corresponding parameter on well-known graph families. In this section, we determine the $1$-limited domination number of Sierpi\'nski graphs. We then focus on the planar family $S(n,3)$ and determine its $k$-limited domination number for both possible values of $k$.

\begin{theorem}\label{sierpinskik1}
Let $n\ge 1$ and $m\ge 2$. Then $$\gamma_1^L(S(n,m))=(m-1)\cdot m^{n-1}.$$
\end{theorem}

\begin{proof}
If $n=1$, then $S(1,m)\cong K_m$, and therefore $\gamma_1^L(S(1,m))=m-1$. Hence, we may assume that $n\ge 2$. 

We first show the upper bound. Let $D=\{x_1x_2\dots x_n\in V(S(n,m)) \, ; \, x_{n-1}\neq x_n \}$. Then, in each copy of the base clique in $S(n,m)$, exactly one vertex does not belong to $D$. Hence $|D|=(m-1)\,m^{n-1}$. We show that $D$ is a $1$-limited dominating set of $S(n,m)$. Clearly, every vertex not belonging to $D$ has a neighbor in $D$ in the same copy of the base clique, so $D$ is a dominating set of $S(n,m)$. Let now $u\in D$. Since exactly one vertex from the copy of the base clique containing $u$ does not belong to $D$, $u$ has exactly one neighbor outside $D$ in this copy. Moreover, by the definition of Sierpi\'nski graphs, if a linking neighbor $u'$ of $u$ exists, then $u'\in D$. Indeed, every linking edge of $S(n,m)$ is between vertices $xijj\dots j$ and $xjii\dots i$ for some prefix $x$ (possibly empty) and distinct symbols $i$ and $j$. Hence, if one endpoint of such an edge belongs to $D$, then its last two coordinates are distinct, and the same is true for the other endpoint. Thus, the other endpoint also belongs to $D$.
Therefore $|N(u)-D|\le 1$. Thus, $D$ is a $1$-limited dominating set of $S(n,m)$, and consequently $\gamma_1^L(S(n,m))\le (m-1)\,m^{n-1}$. 

Next we show the lower bound. Let $D$ be a $1$-limited dominating set of $S(n,m)$. Every base clique in $S(n,m)$ contains at least one vertex from $D$. Indeed, every base clique contains a unique vertex that is not incident with any linking edge. If a base clique contained no vertex from $D$, then this vertex would have no neighbor in $D$, contradicting the fact that $D$ is a dominating set. Moreover, since every base clique in $S(n,m)$ is isomorphic to $K_m$, at most one vertex of a base clique can lie outside $D$. Indeed, if two vertices $u,v\notin D$ belong to the same base clique, then for any vertex $w\in D$ in that clique we have $|N(w)-D|\ge 2$, which contradicts the $1$-limited constraint. Hence, each base clique of $S(n,m)$ contains at least $m-1$ vertices from $D$. Since $S(n,m)$ contains $m^{n-1}$ base cliques, we obtain $|D|\ge (m-1)\,m^{n-1}$. Therefore, $\gamma_1^L(S(n,m))\ge (m-1)\,m^{n-1}$. This completes the proof.
\end{proof}

We now focus on the planar family $S(n,3)$ and determine its $k$-limited domination number.

\begin{theorem}\label{sierpinskisn3}
Let $n\ge 1$ and $k\in\{1,2\}$. Then $$\gamma_k^L(S(n,3))=(3-k)\cdot 3^{n-1}.$$
\end{theorem}

\begin{proof}
For $k=1$, the result follows immediately from Theorem~\ref{sierpinskik1} by setting $m=3$. Indeed,
$\gamma_1^L(S(n,3))=(3-1)\cdot 3^{n-1}=2\cdot 3^{n-1}$.

It remains to consider the case $k=2$. If $n=1$, then $S(1,3)\cong K_3$, and therefore $\gamma_2^L(S(1,3))=1$. Hence, we may assume that $n\ge 2$.

To prove the upper bound, let $D=\{x_1x_2\dots x_n\in V(S(n,3)) \, ; \, x_{n-1}=x_n \}$.
Then $|D|=3^{n-1}$. We show that $D$ is a $2$-limited dominating set of $S(n,3)$. Since every base clique of $S(n,3)$ is isomorphic to $K_3$, the one vertex from that clique that is in $D$ dominates the other two, so $D$ is a dominating set of $S(n,3)$. Let now $u\in D$. Then $u$ has two neighbors in the same base clique that are outside $D$. Moreover, by the definition of Sierpi\'nski graphs, if a linking neighbor $u'$ of $u$ exists, then $u'\in D$. Indeed, every linking edge of $S(n,3)$ is between vertices $xijj\dots j$ and $xjii\dots i$ for some prefix $x$ (possibly empty) and distinct symbols $i$ and $j$. Hence, if one endpoint of such an edge belongs to $D$, then its last two coordinates are equal, and the same is true for the other endpoint. Thus, the other endpoint also belongs to $D$. Therefore $|N(u)-D|\le 2$. Thus, $D$ is a $2$-limited dominating set of $S(n,3)$, and consequently $\gamma_2^L(S(n,3))\le 3^{n-1}$.

Next we show the lower bound. Let $D$ be a $2$-limited dominating set of $S(n,3)$. As in the proof of Theorem~\ref{sierpinskik1}, every base clique contains at least one vertex from $D$, since otherwise its unique vertex not incident with any linking edge would not be dominated. Since $S(n,3)$ contains $3^{n-1}$ base cliques, it follows that $|D|\ge 3^{n-1}$. Therefore, $\gamma_2^L(S(n,3))\ge 3^{n-1}$.

Combining the upper and lower bound yields $\gamma_2^L(S(n,3))=3^{n-1}$.
\end{proof}

\begin{figure}[ht!]
\begin{center}
    \begin{subfigure}{0.45\textwidth}
    \centering
        \begin{tikzpicture}[scale=0.7]
            \node [My Style, name=v1, label=below left:$11$ ]   at (0,0) {};
            \node [My Style2, name=v2, label=below left:$12$ ]   at (2,0) {};
            \node [My Style2, name=v3, label=below left:$21$ ]   at (3,0) {};
            \node [My Style, name=v4, label=below left:$22$ ]   at (5,0) {};
            \node [My Style2, name=v5, label=left:$10$ ]   at (1,1.73) {};
            \node [My Style2, name=v6, label=right:$20$ ]   at (4,1.73) {};
            \node [My Style2, name=v7, label=left:$01$ ]   at (1.5,2.56) {};
            \node [My Style2, name=v8, label=right:$02$ ]   at (3.5,2.56) {};
            \node [My Style, name=v9, label=above:$00$ ]   at (2.5,4.29) {};

            \draw (v1)--(v2)--(v3)--(v4)--(v6)--(v8)--(v9)--(v7)--(v5)--(v1);
            \draw (v5)--(v2);
            \draw (v6)--(v3);
            \draw (v7)--(v8);
        \end{tikzpicture}
    \caption{A $1$-limited dominating set.}
    \label{fig:a}
    \end{subfigure}
    \hfill
    \begin{subfigure}{0.45\textwidth}
    \centering
        \begin{tikzpicture}[scale=0.7]
            \node [My Style2, name=v1, label=below left:$11$ ]   at (0,0) {};
            \node [My Style, name=v2, label=below left:$12$ ]   at (2,0) {};
            \node [My Style, name=v3, label=below left:$21$ ]   at (3,0) {};
            \node [My Style2, name=v4, label=below left:$22$ ]   at (5,0) {};
            \node [My Style, name=v5, label=left:$10$ ]   at (1,1.73) {};
            \node [My Style, name=v6, label=right:$20$ ]   at (4,1.73) {};
            \node [My Style, name=v7, label=left:$01$ ]   at (1.5,2.56) {};
            \node [My Style, name=v8, label=right:$02$ ]   at (3.5,2.56) {};
            \node [My Style2, name=v9, label=above:$00$ ]   at (2.5,4.29) {};

            \draw (v1)--(v2)--(v3)--(v4)--(v6)--(v8)--(v9)--(v7)--(v5)--(v1);
            \draw (v5)--(v2);
            \draw (v6)--(v3);
            \draw (v7)--(v8);
        \end{tikzpicture}
    \caption{A $2$-limited dominating set.}
    \label{fig:b}
    \end{subfigure}
    \hfill
		



\end{center}

\caption{The Sierpi\'nski graph $S(2,3)$ with examples of $1$-limited and $2$-limited dominating sets.}
\label{s23}
\end{figure}
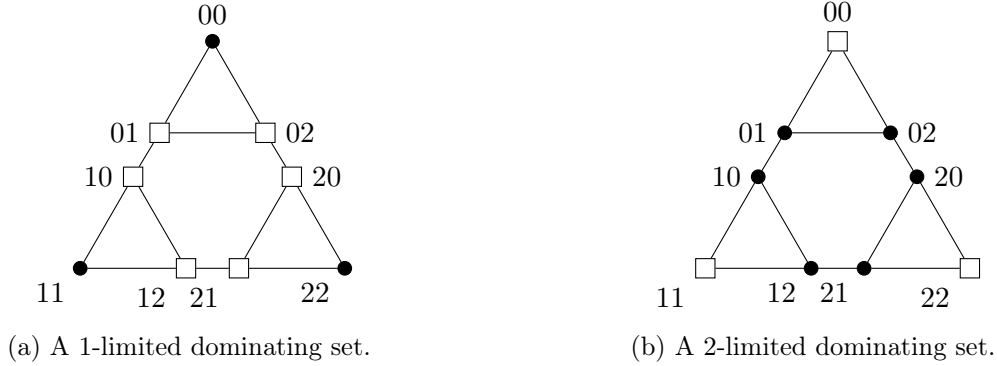

Figure \ref{s23} illustrates the constructions used in the proofs of Theorems \ref{sierpinskik1} and \ref{sierpinskisn3} on the graph $S(2,3)$. The square vertices form a minimum $1$-limited dominating set in Figure~\ref{fig:a} and a minimum $2$-limited dominating set in Figure~\ref{fig:b}.


\section{Conclusion}

In this paper, we prove that the \textsc{$1$-Limited Dominating Set} problem is $\mathsf{NP}$-complete. We also prove that, for every positive integer $k$, the \textsc{$k$-Limited Dominating Set} problem is $\mathsf{NP}$-complete even when restricted to planar graphs. A natural direction for further research is to investigate the complexity of the problem on other restricted graph families. On the structural side, our results on Sierpi\'nski graphs suggest the following conjecture.

\begin{conjecture}
Let $n$, $m$, and $k$ be integers such that $n\ge 1$ and $1\le k\le m-1$. Then
$$\gammakL(S(n,m))=(m-k)\, m^{n-1}.$$
\end{conjecture}

The conjecture, if true, would completely determine the $k$-limited domination number of Sierpi\'nski graphs.


\vskip 1pc \noindent{\bf Acknowledgments.} 
This work has been supported by the European Commission’s Horizon Europe Research and Innovation programme through the Marie Sk{\l}odowska-Curie Actions Staff Exchanges (MSCA-SE) under Grant Agreement no. 101182819 (COVER: \allowbreak\ (C)ombinatorial (O)ptimization for (V)ersatile Applications to (E)merging u(R)ban Problems). 
The author was partially supported by the Slovenian Research and Innovation Agency (ARIS), program P2--0065 and project J1--70016.

\bibliographystyle{plain}
\bibliography{DG-references}

\end{document}